\documentclass[12pt]{article}
\usepackage[utf8]{inputenc}
\usepackage{amsmath,amsthm,amssymb}
\usepackage{mathrsfs}
\usepackage{cite}
\usepackage{algorithm2e}

\usepackage{enumitem}
\usepackage{tikz-cd}
\usepackage{mathtools}
\usepackage{amsfonts}
\usepackage{amscd}
\usepackage{makeidx}
\usepackage{enumitem}

\usepackage{hyperref} 
 \hypersetup{
   colorlinks=true,
   linkcolor=blue,
   filecolor=blue,
   citecolor = black,
   urlcolor=cyan,}

\usepackage{float}

\usepackage{stmaryrd}

\title{An algorithm for computing root multiplicities in Kac-Moody algebras}
\author{Aidan Backus, Peter Connick, and Joshua Lin}
\date{December 2019}

\newcommand{\NN}{\mathbb{N}}
\newcommand{\ZZ}{\mathbb{Z}}

\newcommand{\RR}{\mathbb{R}}

\newcommand{\Knaive}{K_{\text{naive}}}
\newcommand{\Kascent}{K_{\text{ascent}}}

\newcommand{\g}{\mathfrak g}
\newcommand{\h}{\mathfrak h}

\def\XXint#1#2#3{{\setbox0=\hbox{$#1{#2#3}{\int}$ }
\vcenter{\hbox{$#2#3$ }}\kern-.6\wd0}}

\newcommand{\dfn}[1]{\emph{#1}\index{#1}}

\theoremstyle{definition}
\newtheorem{theorem}{Theorem}[section]
\newtheorem{prop}[theorem]{Proposition}
\newtheorem{lemma}[theorem]{Lemma}

\newtheorem{definition}[theorem]{Definition}

\begin{document}

\maketitle

\begin{abstract}
    Root multiplicities encode information about the structure of Kac-Moody algebras, and appear in applications as far-reaching as string theory and the theory of modular functions. We provide an algorithm based on the Peterson recurrence formula to compute multiplicities, and argue that it is more efficient than the naive algorithm.
\end{abstract}

\section{Introduction}
Kac-Moody algebras were introduced in the 1960s by Kac and Moody working independently as generalizations of finite-dimensional semisimple Lie algebras. Every Kac-Moody algebra $\g$ is equipped with a root system $\Delta$ and a Cartan subalgebra $\h$, and we have a root space decomposition
$$\g = \bigoplus_\alpha \g_\alpha \oplus \h \oplus \bigoplus_\alpha \g_{-\alpha}$$
where the direct sums are taken over all \emph{positive} roots $\alpha \in \Delta$. The dimension of the root space $\g_\alpha$ is called the \dfn{multiplicity} of $\alpha$. The root multiplicities of $\g$ encode important information about the structure of $\g$. Closed expressions for the root multiplicities are only known in a few cases, and a major open problem in the field to give general closed expressions, or at least estimates, for root multiplicities. A historical overview of multiplicity theory is given by  \cite[\S1]{carbone2014dimensions}.

In this paper, we discuss an algorithm that implements the Peterson recurrence formula for the root multiplicities. Naively implemented, the Peterson formula iterates over the entire root lattice, whose cardinality grows exponentially in the height of a root; we exploit the semigroup structure of the imaginary roots and divisibility properties of the real roots to greatly cut down on the space of roots one must iterate over.

We have implemented our algorithm in Sage, and at the time of writing are preparing for submission to the Sage Project. Our implementation can also be downloaded \href{https://github.com/catuse/kacmoody/blob/master/kac_moody_algebra.sage}{on GitHub}. It can compute the root multiplicities of the exceptional algebra $E_{10}$ up to height $100$ in a matter of minutes. We compared our computations to the computations for $E_{10}$ and $E_{11}$ given by Kleinschmidt \cite{kleinschmidt2004e11} and for certain hyperbolic Kac-Moody algebras given by Kac \cite[\S11.15]{kac_2014}.

\section{Preliminaries}
Fix a symmetrizable generalized Cartan matrix $A = (a_{ij})_{i,j=1}^d$. We decompose $A = DB$ where $D$ is a diagonal matrix with entries $\varepsilon_1, \dots, \varepsilon_d$ and $B = (b_{ij})_{ij}$ is symmetric. Then $A$ determines a unique Kac-Moody algebra $\g$, equipped with a Cartan subalgebra $\h \subseteq \g$ and an indexed set $\Delta_{simp} = \{\alpha_1, \dots, \alpha_d\} \subset \h^*$ of simple roots. We will always take $\Delta_{simp}$ as the basis of its span, so when we take dot products $\cdot$, they are with $\Delta_{simp}$ as an orthonormal basis.

Let $Q$ denote the root lattice of $\g$; i.e. the lattice in $\h^*$ generated by $\Delta_{simp}$.
\begin{definition}
	The \dfn{multiplicity} $m(\beta)$ of $\beta \in Q$ is the dimension of the vector space $\g_\beta$ of $g \in \g$ such that for every $h \in \h$,
	$$[h, g] = \beta(h)(g).$$
	If $m(\beta) > 0$, we say that $\Delta$ is a \dfn{root} of $\g$ and write $\beta \in \Delta$.
\end{definition}
If $\beta \in \Delta$, then either all coordinates of $\beta$ are positive (i.e. $\geq 0$) or they are all negative. The set of positive $\beta$ is called $\Delta^+$. One has $m(\beta) = m(-\beta)$, so for the purposes of computing root multiplicities, one might as well assume $\beta \in \Delta^+$ (and henceforth we do).

Let $(\cdot,\cdot)$ denote the Killing form of $\g$, so $(\beta, \gamma) = \beta \cdot B\gamma$. Recall that the \dfn{fundamental reflection} by $\alpha_i \in \Delta_{simp}$ is defined by
	$$w_i(\beta) = \beta - (\beta, \alpha_i)\alpha_i,$$
and that the fundamental reflections generate the Weyl group $W$. In particular, $\Delta^r$ is the closure of $\Delta_{simp}$ under fundamental reflections.

We let $\Delta_f$ denote the set of all positive imaginary roots $\beta \in \Delta^i$ such that for every $\alpha_j \in \Delta_{simp}$, $(\beta, \alpha_j) \leq 0$.

For any $\beta \in Q$, we let $|\beta|$ denote the height of $\beta$, i.e. the sum of the coordinates of $\beta$ with respect to $\Delta_{simp}$. We let $\rho$ denote the Weyl vector, so $2(\rho, \beta) = |\beta|$.

\begin{definition}
	A \dfn{divisor} of $\beta$ is a $\gamma \in Q^+$ such that there is a $n \in \NN$ satisfying $n\gamma = \beta$. In this case, we write $\gamma|\beta$.
\end{definition}
With this definition in mind, we define
$$c(\beta) = \sum_{\gamma_p|\beta} \frac{m(\gamma_p)}{p}$$
where we have $d\gamma_p = \beta$. This sum appears in the Peterson recurrence formula. We let $\gcd \beta$ denote the $\gcd$ of the coordinates of $\beta$ (with respect to $\Delta_{simp}$).

\begin{definition}
	Let $\beta \in \Delta_f$. If $\gamma \in Q^+$ is such that $\beta - \gamma \in Q^+$, then we say that $\gamma$ is a \dfn{subroot} of, or is \dfn{under}, $\beta$, and write $\gamma \prec \beta$.
\end{definition}
	It is immediate that $\preceq$ is a partial order, and that if $\gamma \prec \beta$, then $|\gamma| < |\beta|$.

To compute $c(\beta)$, we use Peterson's recurrence formula.
\begin{theorem}[Peterson's recurrence formula]
	One has
	$$c(\beta) = \frac{1}{( \beta, \beta - 2\rho)}\sum_{\gamma \prec \beta} ( \gamma, \beta - \gamma ) c(\gamma) c(\beta - \gamma).$$
\end{theorem}
Peterson's recurrence formula is proven, for example, in Kac's book \cite{kac_2014}.

We will also need the following theorem of convex geometry, proven for example in Bruns-Gubeladze \cite{bruns2009polytopes}.
\begin{theorem}[Gordan]
	Let $\Gamma$ be a rational convex polyhedral cone in $\RR^d$ with dual cone
	$$\Gamma^* = \{y \in \RR^d: \forall x \in \Gamma ~y\cdot x \geq 0\}.$$
	If $G = (G, +)$ is the semigroup of lattice points in $\Gamma^*$, then $G$ is finitely generated.
\end{theorem}

\section{The algorithm}
We can use the action of the Weyl group to compute $\Delta^r$ from $\Delta_{simp}$. More specifically, we use the \dfn{pingpong algorithm}.

\begin{algorithm}[ht]
	\KwData{a root $\alpha \in \Delta^+$, a maximum height $h$}
	$\ell := \text{Stack}(\alpha)$\;
	\While{$\ell \neq \emptyset$}{
		$\beta := \text{Pop}(\ell)$\;
		$P := \{w_1(\beta), \dots, w_d(\beta)\}$\;
		$P := \{\gamma \in P: |\gamma| \leq h \text{ and } \gamma \geq 0\}$\;
		\For{$\gamma \in P \setminus \Delta^+$}{
			$\Delta^+ := \Delta^+ \cup \gamma$\;
			Push$(\ell, \gamma)$\;
			$m(\gamma) := m(\alpha)$\;
			$c(\gamma) := c(\alpha)$\;
		}
	}
\caption{The pingpong algorithm.}
\end{algorithm}

The pingpong algorithm will add $w\alpha$ to $\Delta^+$ for every $w \in W$ such that $|w\alpha| \leq h$, along with recording the values of $m(\gamma)$ and $c(\gamma)$ for $\gamma$ in the orbit, which are preserved by the action of the Weyl group. Indeed, let $w = w_{i_1} \dots w_{i_k}$ and assume that $w^\flat = w_{i_2} \dots w_{i_k}$ is such that $w^\flat \alpha$ been added to $\Delta^+$. Then $w\alpha = w_{i_1}w^\flat \alpha$ and so $w\alpha \in P$. Therefore the claim follows by induction.

After initializing each of the $m(\alpha_j) = c(\alpha_j) = 1$, and pingponging each of the $\alpha_j \in \Delta_{simp}$, we have generated all of $\Delta^r$ up to height $h$. We now must generate the imaginary roots $\Delta^i$. Similar to the case of real roots, we simply must choose one root from each orbit, and to this end we compute the imaginary roots from the imaginary fundamental chamber, $\Delta_f$.

\begin{lemma}
    \label{hilbert basis exists}
	$\Delta_f$ is contained in a semigroup which admits a \dfn{Hilbert basis}; i.e. a minimal, finite generating set.
\end{lemma}
\begin{proof}
	Let $\Gamma \subset \h$ be the fundamental chamber of $\g$. Then $\Gamma$ is defined by the inequality $Bx \geq 0$, $\Gamma$ is polyhedral, and rational since the entries of $\g$ are integers. Now $\Delta_f$ is contained in the semigroup of lattice points $G$ of the dual cone of $\Gamma$ \cite[\S 5.8]{kac_2014}. By Gordan's theorem, $G$ is finitely generated, so we take as our Hilbert basis a generating set of minimal cardinality.
\end{proof}

The Hilbert basis $\beta_1, \dots, \beta_k$ of $\Gamma^*$ can be computed efficiently from the Cartan matrix of $\g$ by e.g. the Elliot-MacMahon algorithm \cite{pasechnik2001computing}. In our implementation we use polymake \cite{polymake2000}'s implementation of the Elliot-MacMahon algorithm. From the Hilbert basis, any $\beta \in \Delta_f$ can be written uniquely as a linear combination of the $\beta_j$.

\begin{lemma}
    \label{real lemma}
	Let $\ell = \gcd \gamma$ and assume $(\gamma, \gamma) > 0$. If $\gamma/\ell \in \Delta$, then $c(\gamma) = 1/\ell$. Otherwise, $c(\gamma) = 0$.
\end{lemma}
\begin{proof}
	We write $\gamma = \sum_i \gamma^i \alpha_i$. We first claim that if $w \in W$ and $\ell = 1$, then $\gcd w\gamma = 1$. Indeed, one has
	$$w_j\gamma = \sum_i \gamma^i (\alpha_i - a_{ij} \alpha_j) = \sum_{i \neq j} \gamma^i \alpha_i + (\gamma^j - \sum_i a_{ij} \gamma^i) \alpha_j$$
	and
	$$\gamma \sum_{i\neq j} (d_i + d_j a_{ij}) \gamma^i + d_j \left(\gamma^j - \sum_i a_{ij} \gamma^i\right) = \sum_i d_i \gamma^i = \ell.$$
	Therefore Bezout's theorem implies that $\gcd w\gamma = 1$. From this it follows that if $\gamma \in \Delta$, then $\ell = 1$.

	Since $(\gamma, \gamma) > 0$, there is at most one $\gamma^\flat \in Q$ in the span of $\gamma$ such that $\gamma \in \Delta$. If $\gamma \in \Delta$, then the above argument shows that $c(\gamma) = m(\gamma) = 1$ and $\ell = 1$. Otherwise, since $c(\gamma) > 0$ and $(\gamma, \gamma) > 0$, there is a $\gamma^\flat|\gamma$ with $m(\gamma^\flat) = 1$. Since we then have $\gcd \gamma^\flat = 1$, it follows that $\ell \gamma^\flat = \gamma$, so the claim follows from definition of $c$.
\end{proof}

We now introduce the \dfn{graded ascent algorithm}.

\begin{algorithm}[ht]
	\KwData{a root $\beta \in \Delta_f$, a maximal height $h$}
	$R := 0$\;
	\For{$\gamma \in \Delta: \gamma \prec \beta$}{
	    \If{$(\gamma, \gamma) > 0$}{
	        $S := 0$\;
	        \For{$1 \leq n \leq |\beta|/|\gamma|$}{
	            \If{$\beta \leq n\gamma$}{break\;}
	            $S := S + ((\gamma, \beta) - n(\gamma, \gamma))c(\gamma)$\;
	        }
	        $R := R + S$\;
	    }
	    \Else{\If{$c(\beta - \gamma) > 0$}{
	        $R := R + ((\gamma, \beta) - (\gamma, \gamma))c(\gamma)$\;
	    }}
	}
	$c(\beta) := R/((\beta, \beta) - 2|\beta|)$\;
	$\Delta := \Delta \cup \beta$\;
	pingpong$(\beta, h)$\;
\caption{The graded ascent algorithm.}
\end{algorithm}
\begin{theorem}
    \label{graded ascent}
    Let $\beta \in \Delta_f$ and $h \geq 0$. Suppose that:
\begin{enumerate}
    \item For every $\alpha \in \Delta_f$ such that $\alpha \prec \beta$, we have already computed $c(\alpha)$ using the graded ascent algorithm with $h$ as input.
    \item We have run the pingpong algorithm on the simple roots and a Hilbert basis.
    \item $|\beta| \leq h$.
\end{enumerate}
    Then for every $w \in W$ such that $|w\beta| \leq h$, the graded ascent algorithm correctly computes $c(w\beta)$.
\end{theorem}
\begin{proof}
    Let $\gamma \in Q$, and suppose $\gamma \prec \beta$. If we can show that either $c(\gamma)$ was already computed, or that the graded ascent algorithm will correctly compute $c(\gamma)$, then the correctness of $c(\beta)$ will follow by the Peterson recurrence formula.

    Suppose $(\gamma, \gamma) > 0$ and let $\ell = \gcd \gamma$. By Lemma \ref{real lemma}, then either $\gamma/\ell$ is a root and $c(\gamma) = 1/\ell$, or $\gamma/\ell$ is not a root and $c(\gamma) = 0$.

    First suppose $\ell = 1$. If $\gamma$ is a root, then there is a simple root $\alpha_j$ and a $w \in W$ such that $\gamma = w\alpha_j$. So the pingpong algorithm correctly placed $\gamma$ in $\Delta$, and $c(\gamma) = 1$. Moreover, every multiple of $\gamma$ which lies under $\gamma$ is of the form $n\gamma$ for $n \in \{1, \dots, |\beta|/|\gamma|\}$. Iterating over such $n$, we also compute the $c(n\gamma)$ correctly. The variable $S$, after iterating over $n$, the sum of the contributions of the $c(n\gamma)$.

    Now suppose $\ell > 1$. Then the contribution of $c(\gamma)$ will be added to $R$ with the contribution of $\gamma/\ell$. In this case, $\gamma \notin \Delta$, so the algorithm does not double-count.

    Now suppose $(\gamma, \gamma) \leq 0$. If $c(\beta - \gamma) = 0$, then $\gamma$ does not contribute to the Peterson recurrence formula and can be neglected. Otherwise, we need to show that the graded ascent algorithm already placed $\gamma$ in $\Delta$. But this follows by the assumption that this is true if $\gamma \in \Delta_f$, by the pingpong algorithm. Indeed, $\gamma = w\gamma_0$ for some $\gamma_0 \in \Delta_f$ and $w \in W$ by basic properties of Kac-Moody algebras.

    By linearity of $w \in W$, $c(w\beta) = c(\beta)$, so that the algorithm correctly computes $c(w\beta)$ as well, and adds them to $\Delta$ correctly.
\end{proof}

We now outline the structure of a program that would use the pingpong and graded ascent algorithms to compute root multiplicities. First, the program runs the Elliot-MacMahon algorithm to compute the Hilbert basis $\beta_1, \dots, \beta_k$. The program then runs the pingpong algorithm on the simple roots and the Hilbert basis with input height $h$. The program maintains an ordering on $\Delta_f$ by height, and iterates the graded ascent algorithm on $\Delta_f$. This guarantees that the assumptions of Theorem \ref{graded ascent} are met, by induction.

\section{Runtime}
We now study the runtime of the above algorithm.

Let $d$ be the dimension of the Cartan matrix, as above and assume that we want to compute all multiplicities up to height $h$. Let $\omega$ be the complexity exponent for matrix multiplication (so multiplication of two $n \times n$ matrices has runtime $O(n^\omega)$, and this estimate is best possible). It is known that $2 \leq \omega < 2.38$ \cite{williams2014multiplying}. Then the computation of a Killing form $(\cdot, \cdot)$ has runtime $O(d^\omega)$. In particular, computing the action of a fundamental reflection has runtime $O(d^\omega)$.

We assume that initializing a root object has $O(d)$ runtime, as a result of overhead due to copying lists. We also observe that by standard results about the complexity of the Euclidean algorithm, the time needed to compute $\gcd \gamma$ for any $\gamma \in Q^+$ is $O(d\log \max \gamma)$ where $\max \gamma$ is the maximum of the coordinates of $\gamma$ with respect to $\Delta_{simp}$. We assume that inserting and looking up in dictionaries has average-case runtime $O(1)$, as it does in Python.

It is not hard to see, then, that the computations of Killing forms will be the dominant term in the runtime of the algorithm, so we bound the number of Killing form computations.

\begin{prop}
    A naive application of the Peterson recurrence formula requires computation of
    $$\Knaive(h) = 4\left(\binom{h+2d-1}{2d} - \left\lceil \frac{h^d}{d!}\right\rceil\right)$$
    Killing forms.
\end{prop}
\begin{proof}
    The naive algorithm requires that we must iterate over $h^d$ many elements of the positive root lattice $Q^+$, and for each $\beta \in Q$, $\beta = (\beta_1, \dots, \beta_n)$ written in the basis of simple roots, we must iterate over $-1 + \prod_j \beta_j$ many subroots of $\beta$, computing $4$ Killing forms for each subroot. We arrange the set
    $$Q_{d,h} = \{\beta \in Q^+: |\beta| \leq h\}$$
    into a $d$-dimensional grid (here shown in dimension $d = 2$, where the left grid shows the coordinates $(\beta_j)_j$ and the right grid shows the product of the coordinates $\prod_j \beta_j$)
\begin{align*}
\begin{matrix}
    (1, h) &\\
    (1, h - 1) & (2, h - 1)\\
    (1, h - 2) & (2, h - 2) & \ddots \\
    \vdots &\vdots &(\beta_1, \beta_2) &\ddots \\
    (1, 1) &(2, 1) & \cdots && (h, 1)
\end{matrix}
&&
\begin{matrix}
    h &\\
    h - 1 & 2(h - 1)\\
    h - 2 & 2(h-2) & \ddots \\
    \vdots &\vdots &\beta_1\beta_2&\ddots \\
    1 &2 & \cdots && h.
\end{matrix}
\end{align*}
    We then compute the sum $S_{d,h}$ of the entries in the right grid. A direct computation shows that
    $$S_{1,h} = 1 + \dots + h = \binom{h+1}{2}.$$
    Suppose inductively that $S_{d,h} = \binom{h+2d-1}{2d}$. We sum up the rows of the right grid corresponding to $Q_{d+1,h}$, each of which is a multiple of a right grid corresponding to $Q_{d+1,h-j}$, to see that
\begin{align*}
    S_{d+1,h} &= S_{d,h} + 2S_{d,h-1} + \dots + hS_{d,1}\\
    &= \binom{h+2d-1}{2d} + 2\binom{h+2d-2}{2d} + \dots + h\binom{2d}{2d}\\
    &= \left(\binom{h+2d-1}{2d} + \binom{h+2d-2}{2d} + \dots + \binom{2d}{2d}\right)\\
    &~~+ \left(\binom{h+2d-2}{2d} + \dots + \binom{2d}{2d}\right) + \dots + \binom{2d}{2d}\\
    &= \binom{h+2d}{2d+1} + \binom{h+2d-1}{2d+1} + \dots + \binom{2d+1}{2d+1}\\
    &= \binom{h+2d+1}{2d+2} = \binom{h+2(d+1) - 1}{2(d+1)}.
\end{align*}
    Therefore, since $|Q_{d,h}|$ is contained in a $d$-simplex of side length $h$ and hence volume $h^d/d!$,
    $$\sum_{\beta \in Q_{d,h}} -1 + \prod_{j=1}^d \beta_j = -h^d + S_{d,h} = \left\lceil \frac{h^d}{d!}\right\rceil + \binom{h+2d-1}{2d}.$$
    This completes the proof.
\end{proof}

Let $\mathcal O_\beta$ denote the set of those elements $\gamma$ of the Weyl orbit of $\beta \in Q^+$ such that $|\gamma| \leq h$. The pingpong algorithm iterates over all of $\mathcal O_\beta$. For each $\gamma$ in the orbit, each fundamental reflection $w_j(\beta) = (\beta, \alpha_j)$, of which there are $d$, must be computed. So computing $\mathcal O_\beta$ using the pingpong algorithm requires $d|\mathcal O_\beta|$ computations of Killing forms.

Let
$$P_h = \frac{|\{\beta \in \Delta^+: |\beta| = h\}|}{|\{\beta \in Q^+: |\beta| = h\}|}$$
denote the probability that a randomly selected element of the root lattice of height $h$ is actually a root. Note that $P_\infty = \lim_{h \to \infty} P_h$ can actually be computed from the Cartan matrix in many cases, and can be used to approximate $P_h$ well in such cases. In trivial cases, such as $E_9$ and finite-dimensional Lie algebras, $P_\infty = 0$. In general, $P_\infty < 1$, since the $\ZZ$-span of each real root $\alpha$ can only meet $\Delta^+$ at one point, namely $\alpha$ itself, yet if $\mathfrak g$ is infinite-dimensional then there are infinitely many real roots. We similarly define
$$Q_h = \max_{j \leq h} \frac{|\{\beta \in \Delta_f: |\beta| = j\}|}{|\{\beta \in Q^+: |\beta| = j\}|}$$
and $Q_\infty = \lim_{h \to \infty} Q_h$, so that $Q_h \leq P_h$.

Suppose, for simplicity, that all simple roots $\alpha_j$ have the same Killing length, i.e. $(\alpha_j, \alpha_j) = C$; this follows, for example, if the Cartan matrix $A$ is symmetric. We note that $P_h$ is increasing, since $P_hh^d$ counts the number of $\gamma \in Q$ such that $(\gamma, \gamma) \leq C$ and $|\gamma| = h$, and the curve $(\gamma, \gamma) = C$ is a hyperboloid. By definition, $Q_h$ is also increasing. So the limits $P_\infty$, $Q_\infty$ exist by the monotone convergence theorem.

\begin{prop}
    \label{kascent}
    Suppose that all simple roots have the same Killing length. Let $\Kascent(h)$ denote the number of Killing forms needed to compute root multiplicities using the graded ascent algorithm up to height $h$. Then
    $$\Kascent(h) \leq P_hQ_h\Knaive(h) + dP_hh^d.$$
    In particular, for any $h$,
    $$\Kascent(h) \leq P_\infty Q_\infty\Knaive(h) + dP_\infty h^d.$$
\end{prop}
\begin{proof}
    For each $\beta \in \Delta_f$, $K_\beta$ denote the number of Killing forms that the graded ascent algorithm uses to apply the Peterson recurrence formula to compute $c(\beta)$. Then
    $$\Kascent(h) = f(h) + \sum_{j=1}^h \sum_{\substack{\beta \in \Delta_f\\|\beta| = j}} K_\beta,$$
    where $f(h)$ is the number of Killing form computations used by the pingpong algorithm. Moreover, $K_\beta$ is at most $4$ times the set of $\gamma \in \Delta$ such that $\gamma \prec \beta$. The set of $\gamma \in \Delta$ such that $|\gamma| = j$ has cardinality $P_j\binom{j+d-1}{d-1}$ since each such $\gamma$ corresponds to a way of summing $d$ natural numbers up to $j$. For the same reason the set of all such $\beta$ has cardinality $Q_h\binom{h+d-1}{d-1}$, so
\begin{align*}
    \Kascent(h)
        &\leq \sum_{j=1}^h P_j \binom{j+d-1}{d-1} Q_h \binom{h+d-1}{d-1}\\
        &\leq P_h Q_h\sum_{j=1}^h \binom{j+d-1}{d-1}\binom{h+d-1}{d-1}\\
        &= P_hQ_h \Knaive(h) \leq P_\infty Q_\infty \Knaive(h).
\end{align*}
    Here we use the fact that
    $$\sum_{j=1}^h \binom{j+d-1}{d-1}\binom{h+d-1}{d-1} \leq \Knaive(h),$$
    which follows from the ``right grid" computation of $\Knaive(h)$. We also use the fact that $P_j,Q_j$ are increasing sequences.

    Pingponging $\beta$ then requires $d|\mathcal O_\beta|$ computations of Killing forms. Since every $\beta \in \Delta^+$ such that $|\beta| \leq h$ will appear in exactly one Weyl orbit computed this way, the sum $\sum_\beta |\mathcal O_\beta| = P_hh^d$ where $\beta$ ranges over representatives of each Weyl orbit. So a total of $f(h) = dP_hh^d$ Killing forms must be computed to execute the pingpong algorithm.
\end{proof}
    Note that for $h$ very large, both our algorithm and the naive algorithm have super-exponential runtime in $d$. However, the coefficient on the leading-order terms are quite different, which can make the difference between minutes and hours' worth of computation in practice. In addition, the estimate in Proposition \ref{kascent} is general enough to hold for any symmetric Cartan matrix, it is rarely sharp. To illustrate, we compute $\Knaive$ and $\Kascent$ for some small values for the exceptional algebra $E_{10}$.

\begin{table}[h!]
\centering
 \begin{tabular}{c | c | c}
 $h$ & $\Knaive(h)$ & $\Kascent(h)$ \\ [0.5ex]
 \hline
 $10$ & $8.218 \cdot 10^9$ & $950$ \\
 $30$ & $9.488 \cdot 10^{19}$ & $4490$\\
 $60$ & $3.299 \cdot 10^{28}$ & $35451$ \\
 $93$ & $4.407 \cdot 10^{34}$ & $696021$
 \end{tabular}
\end{table}
    But $E_{10}$ has few imaginary roots of height $\leq 100$, and has a high dimension, so this is an extreme example. To illustrate a much less extreme case, we consider the Kac-Moody algebra with Cartan matrix $\begin{bmatrix}2 & -3 \\ -3 & 2\end{bmatrix}$, which is more typical of Kac-Moody algebras in practice. Since this algebra is a symmetric algebra whose Cartan matrix has dimension $2$, we only have to compute the subroots $\gamma = \gamma^1 \alpha_1 + \gamma^2 \alpha_2$ of a $\beta \in \Delta_f$ for which $\gamma^1 < \gamma^2$, and then multiply the total multiplicity by $2$ -- an easy optimization which will save us some computation time.
\begin{table}[h!]
\centering
 \begin{tabular}{c | c | c}
 $h$ & $\Knaive(h)$ & $\Kascent(h)$ \\ [0.5ex]
 \hline
 $10$ & $2660$ & $236$ \\
 $20$ & $34620$ & $1719$\\
 $30$ & $161880$ & $6556$ \\
 $40$ & $490440$ & $18079$ \\
 $50$ & $1116300$ & $40883$\\
 $100$ & $17665100$ & $566541$
 \end{tabular}
\end{table}
    So the algorithm is approximately a constant (which is highly significant for practical purposes!) times faster than the naive algorithm, exactly what Proposition \ref{kascent} predicts.

\section*{Acknowledgements}
The authors gratefully acknowledge the financial backing of the Sherrill Fund and the Summer Undergraduate Research Fellowship. 

\bibliography{algorithm}{}

\begin{thebibliography}{1}

\bibitem{bruns2009polytopes}
W.~Bruns and J.~Gubeladze.
\newblock {\em Polytopes, Rings, and K-Theory}.
\newblock Springer Monographs in Mathematics. Springer New York, 2009.

\bibitem{carbone2014dimensions}
Lisa Carbone, Walter Freyn, and Kyu-Hwan Lee.
\newblock Dimensions of imaginary root spaces of hyperbolic kac-moody algebras.
\newblock {\em Contemp. Math}, 623:23--40, 2014.

\bibitem{polymake2000}
Ewgenij Gawrilow and Michael Joswig.
\newblock {\tt polymake}: a framework for analyzing convex polytopes.
\newblock In {\em Polytopes---combinatorics and computation ({O}berwolfach,
  1997)}, volume~29 of {\em DMV Sem.}, pages 43--73. Birkh\"auser, Basel, 2000.

\bibitem{kac_2014}
Victor~G. Kac.
\newblock {\em Infinite-dimensional Lie algebras: an introduction}.
\newblock Birkhauser, 2014.

\bibitem{kleinschmidt2004e11}
Axel Kleinschmidt.
\newblock E11 as e10 representation at low levels.
\newblock {\em Nuclear physics. B}, 677(3):553--586, 2004.

\bibitem{pasechnik2001computing}
Dmitrii~V Pasechnik.
\newblock On computing hilbert bases via the elliot--macmahon algorithm.
\newblock {\em Theoretical computer science}, 263(1-2):37--46, 2001.

\bibitem{williams2014multiplying}
Virginia~Vassilevska Williams.
\newblock Multiplying matrices in o(n\^2.373) time.
\newblock 2014.

\end{thebibliography}
\bibliographystyle{plain}

\end{document}